\documentclass[a4paper, 10pt]{amsart}

\usepackage{latexsym, amssymb,amsmath, amsthm,amsfonts,color,tikz,lscape,geometry}
\usetikzlibrary{math}
\usetikzlibrary{arrows.meta}
\usepackage{pdflscape}
\usepackage{rotating}

\newcommand{\K}{\mathbb{K}}

\newcommand{\M}{{\mathcal{M}}}
\newcommand{\B}{{\mathcal{B}}}

\newcommand{\bA}{{\mathbf{A}}}

\def\SL{\operatorname{SL}}

\def\GL{\operatorname{GL}}

\def\Z{\mathbb{Z}}
\def\N{\mathbb{N}}
\def\C{\mathbb{C}}

\newtheorem{Lemma}{Lemma}[section]
\newtheorem{Theorem}[Lemma]{Theorem}

\newtheorem{cor}[Lemma]{Corollary}

\theoremstyle{definition}

\theoremstyle{remark}

\newtheoremstyle{Acknowledgments}
  {}
    {}
     {}
     {}
    {\bfseries}
    {}
     {.5em}
     {\thmname{#1}\thmnumber{ }\thmnote{ (#3)}}
\theoremstyle{Acknowledgments}

\title{Semi-invariants of a matrix and a covector}

\author{Jonathan Elmer}
\address{Middlesex University\\
The Burroughs, Hendon, London\\
NW4 4BT UK}
\email{j.elmer@mdx.ac.uk}

\date{\today}

\begin{document}

\maketitle

 \begin{abstract}
We prove the following theorem: let $\M_d$ denote the set of $d \times d$ matrices over an infinite field $K$, and let ${(K^d)^*}$ be the set of row vectors. Define an action of $\SL_d(K)$ on $X:= \M_d \oplus (K^d)^*$ by
\[ g \cdot (A,\phi) = (gAg^{-1}, \phi g^{-1}).\]
Then $K[X]^{\SL_d}$ is a polynomial ring, generated by the coefficients of the characteristic polynomial of $A$ and one further invariant, namely $$\Delta(A,\phi):= \det(\phi,\phi A,\phi A^2,\ldots, \phi A^{d-1})^t.$$

Our proof is entirely classical in nature, but we give an interpretation of the result and its proof in terms of quiver representation theory.
\end{abstract}

\section{Introduction}
Let $K$ be an infinite field, $G$ a linear algebraic group over $K$ and $X$ a $G$-variety (on which $G$ acts on the left). Denote by $K[X]$ the algebra of regular functions on $X$. If $X$ is a vector space we may identify $K[X]$ with $S(\mathcal{X}^*)$.  We may define a left action of $G$ on $K[X]$ as follows:
\begin{equation} g \cdot f(v) = f(g^{-1} \cdot v) \end{equation}
for all $g \in G$ and $v \in X$. The fixed points of this action are precisely the regular functions constant on $G$-orbits in $X$. These form a subalgebra $K[X]^G$ of $K[X]$ called the algebra of {\it $G$-invariants}.

More generally, suppose $\theta:G \rightarrow K^*$ is a multiplicative character. We define
\[K[X]^G_\theta = \{f \in K[X]: g \cdot f = \theta(g) f \ \text{for all} \ g \in G\}.\]
It is easy to see that $K[X]^G_\theta$ is a $K[X]^G$-module. Elements of $K[X]^G_{\theta}$ are called {\it isobaric $G$-semi-invariants of weight $\theta$}. 

Now let $H:= [G,G] \leq G$. Clearly every character of $G$ is trivial on $H$. One can show (see, e.g. \cite[Lemma~9.4.1]{DerksenWeyman}) that we have a decomposition of $K[X]^G$-modules
\begin{equation}\label{weightdecomp} K[X]^H = \bigoplus_{\theta \in \mathcal{X}^*(G)} K[X]^{G}_\theta, \end{equation} where $\mathcal{X}^*(G)$ is the set of characters of $G$. 
In this context, $H$-invariants are often called $G$-semi-invariants. 

Before stating our main result we fix some notation. Let $d \geq 1$ and let $G:= \GL_d(K)$, then $H:= [G,G] = \SL_d(K)$. By the terms invariants, semi-invariants, etc, we mean $G$-invariants, $G$-semi-invariants, etc. One can show that every character of $G$ takes the form $\theta_m(g) = \det(g)^m$ for some $m \in \Z$, and we write $K[X]^G_m:= K[X]^G_{\theta_m}$ to simplify notation. We will call elements of $K[X]^G_m$ isobaric semi-invariants of weight $m$.

Let $\M_d$ denote the set of $d \times d$ matrices over $K$ and let $(K^d)^*$ be the set of row vectors. A left action of $G$ on $\M_d$ is given by
\[g \cdot A= gAg^{-1},\]
and a left action of $G$ on $(K^d)^*$ is given by
\[g \cdot \phi = \phi g^{-1}.\]
We fix $X = \M_d \oplus (K^d)^*$, and consider the diagonal action of $G$ on this space, so
\begin{equation}\label{action} g \cdot (A,\phi) = (gAg^{-1},\phi g^{-1}).\end{equation}

Recall that the characteristic polynomial of a matrix $A \in \M_d$ is
\[\chi_A(t) = \det(tI_d-A) \in K[t].\]
This is a monic polynomial of degree $d$. Since $$\chi_{gAg^{-1}}(t) = \det(tI_d-gAg^{-1}) = \det(g) \det(tI-A)\det(g^{-1}) = \chi_A(t),$$ for all $g \in G$, the function $c_i: \M_d \rightarrow K$ which evaluates the coefficient of $t^i$ in $\chi_A(t)$ is  $G$-invariant, for all $i=0, \ldots, d-1$. It is well-known that
\begin{equation}\label{matrixinvar} K[\M_d]^{G} = K[c_0,c_1,\ldots, c_{d-1}],\end{equation}
see for instance \cite[Example~2.1.3]{DerksenKemper}.
We will abuse notation slightly, defining $c_i \in K[X]^G$ by $c_i(A, \phi):= c_i(A)$. We will show in the next section that one also has $$K[X]^G = K[c_0,c_1,\ldots, c_{d-1}].$$

We now define a polynomial function $\Delta \in K[X]$ by
\begin{equation}\label{Delta} \Delta(A,\phi):= \det(\phi,\phi A,\phi A^2,\ldots, \phi A^{d-1})^t. \end{equation}
In the next section we will show that $\Delta$ is an isobaric semi-invariant with weight $1$. 

The main purpose of this article is to prove the following result:
\begin{Theorem}\label{main}
The algebra of semi-invariants $K[X]^H$ is a polynomial ring generated by $c_0,c_1, \ldots, c_{d-1}$ and $\Delta$.
\end{Theorem}

The proof of this result is given in Section 2. In Section 3 we will interpret the result and its proof in the language of quiver representation theory.

\section{Main result}

We begin by proving, as already stated:
\begin{Lemma}
$\Delta$ is an isobaric semi-invariant of weight $1$. In particular, $\Delta \in \C[X]^H$.
\end{Lemma} 

\begin{proof}
Let $g \in G$. Then we have \[g \cdot \phi A^i = \phi g^{-1} (g A g^{-1})^i = \phi g^{-1} g A^i g^{-1} = \phi A^{i} g^{-1}.\] Therefore
\[g \cdot \Delta(A, \phi) = \det( \phi g,  \phi A g, \phi A^2 g,\ldots, \phi A^{d-1} g ) = \Delta \det(g).\]
as required.
\end{proof}

\begin{cor} The set $\{c_0,c_1, \ldots, c_{d-1},\Delta\} \subset K[X]^H$ is algebraically independent.
\end{cor}

\begin{proof}
Suppose $F(c_0,c_1, \ldots, c_{d-1}, \Delta) = 0$ for some polynomial $F$ over $K$ in $d+1$ variables. We may write
\[F(c_0,c_1, \ldots, c_{d-1}, \Delta) = \sum_{i=0}^k \Delta^i p_i(c_0,c_1, \ldots, c_{d-1}) = 0\] for some $k$ and some polynomials $p_i: i=0 \ldots, k$ over $K$ in $d$ variables.
Each summand $\Delta^i p_i$ is isobaric of (different) weight $i$. By \eqref{weightdecomp} each summand must be zero. Hence $p_i(c_0, c_1, \ldots, c_{d-1})=0$ for all $i$. By \eqref{matrixinvar}, $c_0, \ldots, c_{d-1}$ are algebraically independent which shows that $p_i$ is identically zero for all $i$, hence $F$ is identically zero.
\end{proof}

We will also need the following:
\begin{Lemma}\label{posdeg}
$K[X]^H$ contains no nonzero isobaric semi-invariants of negative weight. In addition, $K[X]^G = K[\M_d]^G$.
\end{Lemma}

\begin{proof}
Let $x_1,x_2, \ldots, x_d$ be the coordinate functions on $(K^d)^*$. Any nonzero $f \in K[X]$ can be expressed as $\sum_{\mathbf{e} \in \N^d} f_{\mathbf{e}} \mathbf{x}^\mathbf{e}$ where $\mathbf{x}^\mathbf{e} = \prod_{i=1}^d x_i^{e_i}$ and $f_{\mathbf{e}} \in K[\M_d]$. Suppose $f \in K[X]^G_m$ and let $g= tI_d$.  Then $$t^{md} f = g \cdot f =  \sum_{\mathbf{e} \in \N^d} t^{|\mathbf{e}|} f_{\mathbf{e}} \mathbf{x}^\mathbf{e}$$ where $|\mathbf{e}| = \sum_{i=1}^d e_i$. Since this holds for all $t \in K^*$, we must have $f_{\mathbf{e}} = 0$ for all $\mathbf{e}$ where $|\mathbf{e}| \neq md$. In particular, we must have $m \geq 0$ as required. Further, if $m=0$ then $f_{\mathbf{e}} = 0$ whenever $\sum_{i=1}^d e_i >0$, so $f \in K[\M_d]$. Since in this case $f$ is $G$-invariant, we have $f \in K[\M_d]^G$ as required.
\end{proof}

We are now ready to prove our main result:

\begin{proof}[Proof of Theorem \ref{main}]
Let $R = K[c_0, c_1, \ldots, c_{d-1}, \Delta]$. We have shown that $R$ is a polynomial ring and $R \subseteq K[X]^H$; only the reverse inclusion remains to be shown.

Consider the set $$X_s = \{(A, \phi) \in X: \Delta(A, \phi) \neq 0\}.$$ Then $X_s$ is open and dense in the irreducible variety $X$. Moreover, $(A, \phi) \in X_s$ if and only if $\B_{A, \phi}:=\{\phi, \phi A, \ldots, \phi^{d-1} A\}$ is a basis of $(K^d)^*$.

Suppose $(A, \phi) \in X_s$. Then there exist unique constants $b_0, \ldots, b_{d-1}$ such that
\[\phi A^d = \sum_{i=0}^{d-1} b_i \phi A^i.\]
On the other hand, by the Cayley-Hamilton theorem we have
\[A^d = -\sum_{i=0}^{d-1} c_i(A) A^i.\]
Applying $\phi$ on the left shows immediately that $b_i=-c_i(A)$ for all $i$.

By considering the action of $A$ on the basis $\B_{A, \phi}$, we see that there exists $g \in G$ such that
\[g \cdot (A, \phi) = \left(\begin{pmatrix} 0 & 0 & \cdots & 0 & -c_{0}(A) \\ 1 & 0 & \cdots & 0 & -c_1(A)\\
0 & 1 & \cdots &0 & -c_2(A)\\ \vdots & \vdots & \cdots & \vdots & \vdots \\ 0 & 0 & \cdots & 1& -c_{d-1}(A)  \end{pmatrix}, (1,0,0,\ldots, 0) \right).\]

Therefore, we see that $(A,\phi), (B, \psi) \in X_s$ belong to the same $G$-orbit if $c_i(A)=c_i(B)$ for all $i=0, \ldots, d-1$.
Conversely, if $c_i(A)=c_i(B)$ for all $i=0, \ldots, d-1$, we may choose an element $g \in G$ satisfying $\phi A^i g^{-1} = \psi B^i$ for all $i=0, \ldots, d-1$. Then we also have
\[\phi A^d g^{-1} = -\sum_{i=0}^{d-1} c_i(A) \phi A^i g^{-1} =  -\sum_{i=0}^{d-1} c_i(B) \psi B^i = \psi B^d.\]
Therefore
\[\phi A^{i} A g^{-1} = \psi B^i B = \phi A^i g^{-1} B\]
so \[\phi A^i A  = \phi A^i (g^{-1} B g)\] for all $i=0, \ldots, d-1$. Since $A$ and = $g^{-1}Bg$ have the same image on every element of the basis $\B_{A, \phi}$, we conclude that $A = g^{-1}Bg$. Therefore
\[g \cdot (A, \phi) = (B, \psi),\] i.e. $(A, \phi)$ and $(B, \psi)$ belong to the same $G$-orbit.
 
Define a morphism $\pi: X_s \rightarrow K^d$ by
\[\pi(A, \phi) = (c_0(A), c_1(A), \ldots, c_{d-1}(A)).\]
The fibres of $\pi$ are precisely the orbits in $X_s$. 

Now suppose $f \in K[X]^H$ is isobaric of weight $m$. By Lemma \ref{posdeg}, we have $m \geq 0$. Consider the restriction $\hat{f}$ of $f$ to $X_s$, which is also isobaric of weight $m$. Since $K[X_s] = K[X]_{\Delta}$, we have $F:=\Delta^{-m}\hat{f} \in K[X_s]$. Moreover, $F$ is isobaric of weight zero, so $F \in K[X_s]^G$. Therefore $F$ is constant on the fibres of $\pi$ and $F = s \circ \pi$ for some morphism $s:K^d \rightarrow K$. This shows that $F = p(c_0, \ldots, c_{d-1})$ as functions on $X_s$, for some polynomial $p$ in $d$ variables. It follows that
$\hat{f} = \Delta^m p(c_0, \ldots, c_{d-1})$ as functions on $X_s$. Since $X_s$ is dense we must have that $f =  \Delta^m p(c_0, \ldots, c_{d-1})$ as functions on $X$. Now by \eqref{weightdecomp} we have $f \in R$ for all $f \in K[X]^H$ as required.
 \end{proof}

\section{Motivation and context}

Recall that a {\it quiver} is a quadruple $Q = (Q_v,Q_a,t,h)$, consisting of two ordered sets $Q_v$ (vertices)  and $Q_a$ (arrows), along with two functions $t, h: Q_a \rightarrow Q_v$ (tail and head respectively). It is usually visualised as a directed graph with a node for each element of $Q_v$, and for each $a \in Q_a$ a directed edge leading from $t(a)$ to $h(a)$. Assume $Q$ is finite, then we may choose an order on the vertices and on the arrows, setting $Q_v = \{x_1, \ldots, x_k\}$ and $Q_a = \{a_1, \ldots, a_n\}$.

Let $K$ be a field.  A {\it representation} $V$ of the quiver $Q$ over $K$ is an assignment to each vertex $x \in Q_v$ of a vector space $V(x)$, and to each arrow $a \in Q_a$ of a linear map $V(a): V(t(a)) \rightarrow V(h(a))$.  We write $\alpha = (\alpha(x): x \in Q_v) \in \N^k$ where $\alpha(x) = \dim(V(x))$ for all $x \in Q_v$; this is called the {\it dimension vector} of $V$. A homomorphism $\phi: V \rightarrow W$ between representations $V, W$ of $Q$ is a collection of linear maps $(\phi(x): x \in Q_v)$ such that $\phi(x): V(x) \rightarrow W(x)$
for all $x \in Q_v$, and for all $a \in Q_a$ the following diagram commutes:\\

\begin{center}
\begin{tikzpicture}
\node[left] at (0,0){$V(t(a))$};
\node[right] at (3,0){$V(h(a))$};
\node[left] at (0,-2){$W(t(a))$};
\node[right] at (3,-2){$W(h(a))$};
\draw[->] (0,0)--(3,0);
\node[above] at (1.5,0){$V(a)$};
\draw[->] (0,-2)--(3,-2);
\node[above] at (1.5,-2){$W(a)$};
\draw[->] (-0.5,-0.5) -- (-0.5,-1.5);
\node[left] at (-0.5,-1){$\phi(t(a))$};
\draw[->] (3.5,-0.5) -- (3.5,-1.5);
\node[right] at (3.5,-1){$\phi(h(a))$};
\end{tikzpicture}
\end{center}


A homomorphism $\phi: V \rightarrow W$ is an isomorphism if $\phi(x)$ is an isomorphism for each $x \in Q_v$.  We say that $W$ is a subrepresentation of $V$ if $W(x) \leq V(x)$ for all $x \in Q_v$ and the collection of inclusion maps $i(x): W(x) \hookrightarrow V(x)$ is a homomorphism $W \rightarrow V$. We write $W \leq V$ if $W$ is a subrepresentation of $V$.


There is a natural notion of direct sum for representations of a given quiver: if $V$ and $W$ are representations of $Q$ with dimension vectors $\alpha$ and $\beta$ respectively, then $V \oplus W$ is the representation of $Q$ with dimension vector $\alpha+\beta$ defined by
\[(V \oplus W)(x) = V(x) \oplus W(x)\] for all $x \in Q_v$ and
\[(V \oplus W)(a) = V(a) \oplus W(a)\] for all $a \in Q_a$.

A representation is said to be {\it indecomposable} if it is not isomorphic to the direct sum of two non-trivial representations. The Krull-Remak-Schmidt Theorem \cite[Theorem~1.7.4]{DerksenWeyman} states that every representation of $Q$ may be written as a direct sum of indecomposable representations and this decomposition is unique up to reordering the summands. Thus, describing the representations of $Q$ up to isomorphism is reduced to the problem of describing the isomorphism classes of indecomposable representations. The quiver $Q$ is said to have:

\begin{enumerate}
\item {\it finite representation type} if $Q$ has only finitely many isomorphism classes of indecomposable representations;
\item {\it tame representation type} if the indecomposable representations of $Q$ in each dimension vector up to isomorphism occur in finitely many one-parameter families;
\item {\it wild representation type} otherwise.
\end{enumerate}

For any $p,q \geq$ we let $\M_{p,q}$ denote the set of $p \times q$ matrices over $K$. We usually write  $\M_{p}$ for $\M_{p,p}$. Let $V$ be a representation of the quiver. By choosing a basis of each vector space $V(x)$, we may identify $V$ with the $n$-tuple of matrices $$\bA = (A_1, A_2, \ldots, A_n)$$ where $A_j \in \M_{\alpha(h(a_j)),\alpha(t(a_j))}$ for all $j$ is the matrix representing $V(a_j)$ with respect to the chosen basis. Choosing a different basis is tantamount to replacing $\bA$ with 
\begin{equation}\label{quiveraction} g \cdot \bA:= (g_{t(a_1)} A_1 g^{-1}_{h(a_1)},  g_{t(a_2)} A_2 g^{-1}_{h(a_2)}, \ldots,g_{t(a_n)} A_n g^{-1}_{h(a_n)} ) \end{equation} where
$$g = (g_{x_1},g_{x_2}, \ldots, g_{x_k}) \in  \GL_{\alpha}(K):= \prod_{i=1}^k \GL_{\alpha(x_i)}(K)$$ is the $k$-tuple of change of basis matrices where $g_{x_i}$ describes the change of basis on $V(x_i)$. Thus, we have an action of $\GL_{\alpha}(K)$ on 
\[\M_{Q,\alpha}:= \prod_{i=1}^n \M_{\alpha(h(a_i)),\alpha(t(a_i))}\] and a pair of $n$-tuples of matrices $\bA, \bA' \in \M_{Q,\alpha}$ represent isomorphic representations of $Q$ if and only if they lie in the same $\GL_{\alpha}(K)$-orbit.

The algebra of invariants $K[\M_{Q,\alpha}]^{\GL_\alpha(K)}$ associated to the action defined in \eqref{quiveraction} is denoted by $I(Q,\alpha)$. Notice that $[\GL_\alpha(K),\GL_{\alpha}(K)] = \SL_\alpha(K):=  \prod_{i=1}^k \SL_{\alpha(x_i)}(K)$, so it makes sense to call $SI(Q,\alpha):= K[\M_{Q,\alpha}]^{\SL_{\alpha}(K)}$ the algebra of semi-invariants of $(Q, \alpha)$. By \eqref{weightdecomp} we have a decomposition $$SI(Q,\alpha) = \bigoplus_{\theta \in \mathcal{X}^*(\GL_\alpha(K))} SI(Q,\alpha)_{\theta}$$
where $ SI(Q,\alpha)_{\theta} = K[\M_{Q,\alpha}]^{\GL_\alpha(K)}_{\theta}$. One can show that every character of $\GL_{\alpha}(K)$ is of the form
\[\prod_{i=1}^k (\det)^{m_i},\] see \cite[Exercise~9.8.1]{DerksenWeyman}.

Semi-invariants of quivers are closely tied to representation type. A celebrated result of Sato and Kimura \cite{SatoKimura} states that $SI(Q,\alpha)$ is polynomial for every dimension vector $\alpha$ is and only if $Q$ has finite representation type. Along similar lines, Skowro\'{nski} and Weyman \cite{SkowronskiWeyman} showed that $SI(Q,\alpha)$ is a complete intersection for every dimension vector $\alpha$ is and only if $Q$ has finite or tame representation type.

The group action studied in this paper is of the form \eqref{quiveraction}. The action associated with the quiver 
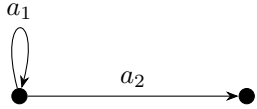
\begin{center}
\begin{tikzpicture}[
    >=Stealth,
    vertex/.style={circle,draw,minimum size=2mm,inner sep=1pt, fill},
    every loop/.style={looseness=50}
]

\node[vertex] (x) at (0,0) {};
\node[vertex] (y) at (3,0) {};

\draw[->] (x) edge[loop above] node {$a_1$} ();
\draw[->] (x) -- node[above] {$a_2$} (y);

\end{tikzpicture}
\end{center}
with dimension vector $(d,1)$ is an action of $\GL_d(K) \times K^*$ on $\M_d \times (K^d)^*$ defined by
\[(g,h) \cdot (A, \phi) = (gAg^{-1},  h \phi g^{-1})\]
for $g \in \GL_d(K), h \in K^*$. The kernel of this action is $D:=\{(tI_d, t): t \in K^*\}$, so this can be regarded as an action of 
$$(\GL_d(K) \times K^*)/D \cong G.$$ We have an algebra isomorphism
\[K[X]^{\GL_d \times K^*} \cong K[X]^G\] and $K[X]^G$ module isomorphisms
\[K[X]^{\GL_d \times K^*}_{m,-md} \cong K[X]_m^G\]
where $K[X]^{\GL_d \times K^*}_{m,-md}$ denotes semi-invariants relative to the character $(g,h) \mapsto \det(g) ^m h^{-md}$. 

This quiver is quite an interesting one. Its representation type is known to be wild. However, it is {\it minimally wild} in the sense that removing a vertex, or contracting an edge always results in a quiver with tame (in fact, finite) representation type. In spite of this, its representations and semi-invariants are not widely studied. I know of only one example, which is a crucial step in the proof of \cite[Theorem~1]{SkowronskiWeyman}: the authors show that $SI(Q,(5,2))$ is not a complete intersection. Presumably they first ruled out the possiblity that $SI(Q,(d,1))$ is not a complete intersection but if they did so they did not publish their result.

Returning to the general situation, let $\theta$ be a nontrivial character of $\GL_{\alpha}(K)$. The kernel of $\theta$ is a subgroup of $\GL_{\alpha}(K)$ containing $\SL_{\alpha}(K)$.  A result of A. King (see \cite[Proposition~9.8.3, Lemma 9.8.5]{DerksenWeyman}) states that the following are equivalent:

\begin{enumerate}
\item $\ker(\theta) \cdot V$ is closed with maximal dimension;
\item $\sum_{i=1}^k m_i\beta_i <0$ for all $\beta \in \N^k$ which are dimension vectors of a proper subrepresentation $U \leq V$.
\end{enumerate}

A representation of $Q$ satisfying these conditions is said to be {\it stable} with respect to $\theta$. For the quiver studied in the present article, the open dense subset $X_s \subseteq X$ which appears in the proof of Theorem \ref{main} is precisely the set of representations which are stable with respect to the character $\theta_1(g,h) = \det(g) h^{-d}$: to see this, note that by (2) above, $(A, \phi)$ is stable with respect to $\theta_1$ if and only if there is no proper subspace $U \subseteq K^d$ satisfying $A(U) \subseteq U$ and $U \subseteq \ker(\phi)$. However, in this instance we also have $$\ker(\theta_1) = \{(g,h) \in \GL_d(K) \times K^*: \det(g) = h^d\},$$ and $\ker(\theta_1)/D \cong \SL_d(K)$, so that the orbits of $\ker(\theta_1)$ are in fact the orbits of $\SL_d(K)$ and (1) implies that $X_s$ is precisely the set of points whose $\SL_d(K)$ orbit is closed with maximal dimension.

In the classical proof of \eqref{matrixinvar} one considers the restriction of invariant functions in  $\K[M_d]$ to the subset of diagonalisble matrices. These are precisely the matrices in $\M_d$ whose $\SL_d(K)$ orbit is closed with maximal dimension. In that sense, our proof of Theorem \ref{main} is analogous to the classical proof of \eqref{matrixinvar}.
\bibliographystyle{plain}
\bibliography{MyBib}

@book {DerksenKemper,
    AUTHOR = {Derksen, Harm and Kemper, Gregor},
     TITLE = {Computational invariant theory},
    SERIES = {Invariant Theory and Algebraic Transformation Groups, I},
      NOTE = {Encyclopaedia of Mathematical Sciences, 130},
 PUBLISHER = {Springer-Verlag},
   ADDRESS = {Berlin},
      YEAR = {2002},
     PAGES = {x+268},
      ISBN = {3-540-43476-3},
   MRCLASS = {13A50 (13P10)},
  MRNUMBER = {MR1918599 (2003g:13004)},
MRREVIEWER = {Dmitri I. Panyushev},
}

@book {DerksenWeyman,
    AUTHOR = {Derksen, Harm and Weyman, Jerzy},
     TITLE = {An introduction to quiver representations},
    SERIES = {Graduate Studies in Mathematics},
    VOLUME = {184},
 PUBLISHER = {American Mathematical Society, Providence, RI},
      YEAR = {2017},
     PAGES = {x+334},
      ISBN = {978-1-4704-2556-2},
   MRCLASS = {16-02 (13A50 14L24 16G10 16G20 16G70)},
  MRNUMBER = {3727119},
MRREVIEWER = {Xueqing Chen},
       DOI = {10.1090/gsm/184},
       URL = {https://doi.org/10.1090/gsm/184},
}

@article {SatoKimura,
    AUTHOR = {Sato, M. and Kimura, T.},
     TITLE = {A classification of irreducible prehomogeneous vector spaces
              and their relative invariants},
   JOURNAL = {Nagoya Math. J.},
  FJOURNAL = {Nagoya Mathematical Journal},
    VOLUME = {65},
      YEAR = {1977},
     PAGES = {1--155},
      ISSN = {0027-7630},
   MRCLASS = {32M10 (20G05)},
  MRNUMBER = {430336},
MRREVIEWER = {A. L. Onishchik},
       URL = {http://projecteuclid.org/euclid.nmj/1118796150},
}

@article {SkowronskiWeyman,
    AUTHOR = {Skowro\'{n}ski, A. and Weyman, J.},
     TITLE = {The algebras of semi-invariants of quivers},
   JOURNAL = {Transform. Groups},
  FJOURNAL = {Transformation Groups},
    VOLUME = {5},
      YEAR = {2000},
    NUMBER = {4},
     PAGES = {361--402},
      ISSN = {1083-4362},
   MRCLASS = {16G20 (13A50 20G05)},
  MRNUMBER = {1800533},
MRREVIEWER = {Olivier G. Schiffmann},
       DOI = {10.1007/BF01234798},
       URL = {https://doi.org/10.1007/BF01234798},
}

\end{document}